\documentclass[11pt,reqno]{amsart}

\usepackage[T1]{fontenc}
\pdfoutput=1
\usepackage{mathptmx,microtype,hyperref,graphicx,dsfont}
\usepackage{amsthm,amsmath,amssymb}
\usepackage[foot]{amsaddr}
\usepackage{enumitem}
\usepackage[nameinlink,capitalise]{cleveref}
\usepackage[font=footnotesize,margin=2cm]{caption}
\usepackage{cite}
\usepackage{framed}
\usepackage{todonotes}

\usepackage{xcolor}

\crefname{theorem}{Theorem}{Theorems}
\crefname{thm}{Theorem}{Theorems}
\crefname{mainthm}{Theorem}{Theorems}
\crefname{lemma}{Lemma}{Lemmas}
\crefname{lem}{Lemma}{Lemmas}
\crefname{remark}{Remark}{Remarks}
\crefname{claim}{Claim}{Claims}
\crefname{subclaim}{Sub-claim}{Sub-claims}
\crefname{prop}{Proposition}{Propositions}
\crefname{proposition}{Proposition}{Propositions}
\crefname{defn}{Definition}{Definitions}
\crefname{corollary}{Corollary}{Corollaries}
\crefname{conjecture}{Conjecture}{Conjectures}
\crefname{question}{Question}{Questions}
\crefname{chapter}{Chapter}{Chapters}
\crefname{section}{Section}{Sections}
\crefname{figure}{Figure}{Figures}
\newcommand{\crefrangeconjunction}{--}

\theoremstyle{plain}

\newtheorem{thm}{Theorem}
\newtheorem*{thm*}{Theorem}
\newtheorem{lemma}[thm]{Lemma}
\newtheorem{lem}[thm]{Lemma}

\newtheorem{prop}[thm]{Proposition}

\theoremstyle{definition}
\newtheorem{defn}[thm]{Definition}

\theoremstyle{remark}

\numberwithin{equation}{section}

\newcommand{\eps}{\varepsilon}
\renewcommand{\P}{{\mathbb P}}
\newcommand{\E}{{\mathbb E}}
\renewcommand{\l}{\langle}
\renewcommand{\r}{\rangle}
\newcommand{\G}{{\mathcal G}}
\newcommand{\Gnp}{\G_{n,p}}
\newcommand{\1}{{\bf 1}}
\newcommand{\cH}{{\mathcal H}}
\newcommand{\cC}{{\mathcal C}}
\newcommand{\cV}{{\mathcal V}}

\author[Z. Bartha]{Zsolt Bartha}
\address{Alfr\'{e}d R\'{e}nyi Institute of Mathematics}
\email{bartha@renyi.hu}

\author[B. Kolesnik]{Brett Kolesnik}
\address{Department of Statistics, University of Oxford}
\email{brett.kolesnik@stats.ox.ac.uk}


\begin{document}

\title[Weakly saturated random graphs]
{Weakly saturated random graphs}

\begin{abstract}
As introduced by Bollob{\'a}s, 
a graph $G$ is weakly $H$-saturated if the complete graph $K_n$ is 
obtained by iteratively completing copies of $H$ minus an edge. 
For all graphs $H$, 
we obtain an asymptotic lower bound for 
the critical threshold $p_c$, at which point 
the Erd{\H{o}}s--R{\'e}nyi graph $\Gnp$ is likely to be weakly $H$-saturated. 
We also prove an upper bound for $p_c$, 
for all $H$ which are, in a sense, strictly balanced. 
In particular, we improve the upper bound 
by Balogh, Bollob{\'a}s and Morris 
for $H=K_r$, 
and we conjecture that this is sharp up to constants. 
\end{abstract}

\maketitle

\section{Introduction}\label{S_intro}

The concept of {\it weak saturation} in graphs was introduced 
by Bollob{\'a}s \cite{B67}. 
Given graphs $G$ and $H$, the graph $\l G\r_H$ is obtained 
by iteratively completing copies of $H$ minus an edge, starting with $G$. 
Formally, set $G_0=G$, and for $t\ge1$, construct $G_t$ by adding every edge
not in $G_{t-1}$ which if added to $G_{t-1}$ creates a new copy of $H$. 
We let $\l G\r_H=\bigcup_{t} G_t$ denote the result of this procedure. 
If $\l G\r_H$ is the complete graph on the vertex set of $G$, that is, if 
all missing edges are eventually added, 
we say that $G$ is {\it weakly $H$-saturated,} or that it {\it $H$-percolates}. 

This process 
can be viewed as a type of cellular automaton  \cite{U50,vN66}, 
of which bootstrap percolation (see, e.g., 
\cite{PRK75,CLR79,vE87,AL88,S92,H03,BBDCM12})
is a well-studied example. 
Balogh, Bollob{\'a}s and Morris \cite{BBM12} 
introduced a random process called {\it graph bootstrap percolation,}
taking $G$ above to be the  
Erd{\H{o}}s--R{\'e}nyi \cite{ER59} graph $\Gnp$. 
The critical point $p_c$, at which $\Gnp$ is likely
to $H$-percolate, is defined formally as 
\[
p_c(n, H)=\inf\{p>0:\P(\l\Gnp\r_{H}=K_n)\ge1/2\}.
\]

The purpose of this work is to obtain 
general bounds for $p_c$. Our first main result 
(see \cref{T_LBgen}) establishes a non-trivial lower
bound, which holds for {\it all} $H$. 
This follows by a general extremal result (see \cref{P_lamstar}) 
that lower bounds the number of edges in so-called ``witness graphs,'' 
which add a given edge. This extends the bound in \cite{BBM12},
proved in the case that $H=K_r$, to all $H$. 
As an application  
(see  \cref{C_Hmain}), 
we locate $p_c$ up to poly-logarithmic factors for all balanced 
(see \cref{D_bal}) graphs $H$, partially
answering Problem 1 in \cite{BBM12}. 
Our second main result (see \cref{T_Hmain})
proves a sharper upper bound for all strictly balanced $H$. 
In particular, this improves the bound 
in \cite{BBM12}
for $K_r$, when $r\ge5$.

The primary focus in \cite{BBM12} is the case that $H=K_r$
is a complete graph (although some other graphs are also analyzed, see 
\cite[Section 5]{BBM12}). 
Note that all graphs $K_2$-percolate (any missing edge is added at time $t=1$) 
so trivially $p_c(n,K_2)=0$. A graph $K_3$-percolates if and only if it is connected, 
so it follows $\P(\l\Gnp\r_{K_3}=K_n)\to \exp(-e^{-c})$ if $p=(\log n +c)/n$
by the fundamental work \cite{ER59}. 
The next threshold of interest
$p_c(n,K_4)$ is estimated in \cite{BBM12} up to constant factors, and the 
recent works
\cite{AK17,AK18,K17} show that 
$p_c(n,K_4)\sim1/\sqrt{3n\log n}$. 

For $r\ge5$, 
$p_c(n,K_r)$ is estimated in \cite{BBM12}
up to poly-logarithmic factors. 
The upper bound  for $p_c(n,K_r)$ proved in \cite{BBM12} holds
for a more general class of graphs, which we now recall.

\subsection{Balanced graphs}
For a graph $H$, we let $v_H$ and $e_H$ denote its 
number of its vertices and edges, respectively, and we put 
\[\lambda=(e_H-2)/(v_H-2).
\]
Note that $H$-percolation is trivial if $v_H\le3$ 
or if the minimum
degree $\delta_H=1$. In the latter
case,
$p_c$ essentially coincides with the threshold for 
a copy of $H$ minus an edge in $\Gnp$ (see \cite[Proposition 26]{BBM12}). That is,  
$p_c=\Theta(n^{-1/\lambda'})$, where
\[ 
\lambda'=\min_{e\in E[H]}\max_{F\subset H\setminus e}e_F/v_F,
\] 
where $E[H]$ is the edge set of $H$, and $F$ is a subgraph of $H$ not containing the edge $e$. 
We therefore assume throughout this work 
that $\delta_H\ge2$ and $v_H\ge4$. In particular, this implies $\lambda\ge1$. 
In fact, these assumptions hold for every graph satisfying the next definition.

\begin{defn}\label{D_bal}
We say that a graph $H$ is {\it balanced} if $v_H\ge 4$, and 
$(e_F-1)/(v_F-2)\le \lambda$
for all subgraphs $F\subset H$
with $3\le v_F<v_H$. 
\end{defn}

This is related to the
notion of a {\it $2$-balanced} graph $G$,  
such that 
$(e_F-1)/(v_F-2)$ is maximized (over $F\subset G$
with $v_F\ge 3$) when $F=G$. 
This concept plays a role in, e.g., \cite{BMS15,CG16,ST15,S16}, 
where the maximal number of edges in an
$H$-free subgraph (Tur\'an's problem) of $\Gnp$ 
is studied. 
Indeed, a graph $H$ is balanced as above if and only if 
$H\setminus e$ is 2-balanced, for all edges $e\in E[H]$.
It also follows that $H$ is connected. 
See \cref{A_supp1} for a proof of these basic facts.

In \cite{BBM12}, it is shown that $p_c(n,K_r)=n^{-1/\lambda+o(1)}$,
as $n\to\infty$. The upper bound holds for 
balanced graphs $H$ (see \cite[Proposition 3]{BBM12}). 
The lower bound, on the other hand,  
relies on the so-called
{\it witness set algorithm,} which assigns to each $e\in E[\l G\r_H]$ 
a {\it witness graph} $W_e\subset G$ such that $e\in E[\l W_e\r_H]$. 
This algorithm yields an Aizenman--Lebowitz \cite{AL88}
type property (see \cref{L_HAL}), a standard
tool from the theory of bootstrap percolation. 
A lower bound for $p_c$ is obtained by the first moment method in  
\cite{BBM12} using this, together with 
the fact (Lemma 9 in \cite{BBM12}) that if $H=K_r$
then a witness graph on $k$ vertices has at least
$\lambda(k-2)+1$ edges. 
The proof, however, 
is somewhat abstract and lengthy. 
The authors state that 
``the proof is delicate, 
and does not seem to extend easily to other graphs.''

In this work, we present a short and simple proof 
(see \cref{P_lamstar} below) 
that works
directly with the $H$-percolation dynamics, and naturally for {\it all} graphs $H$. 
Using this, we obtain the following result, which answers Problem 1 in \cite{BBM12}
in the case that $H$ is balanced. 

\begin{thm}\label{C_Hmain}
If $H$ is balanced (see \cref{D_bal}) then $p_c(n,H)=n^{-1/\lambda+o(1)}$. 
\end{thm}

We note that Bayraktar and Chakraborty \cite{BC19} 
recently studied the complete bipartite case. 
Extending the arguments in \cite{BBM12} for $H=K_r$ to $K_{r,s}$,
they find $p_c$
up to poly-logarithmic factors in the balanced case, partially answering 
Problem 5 in \cite{BBM12}. 
\cref{C_Hmain} contains these results in  \cite{BC19} 
as a special case.

\subsection{General lower bound}
\cref{C_Hmain} follows by the upper bound in \cite{BBM12} and a general lower bound 
for $p_c(n,H)$, that holds for all $H$ (satisfying our
baseline assumptions $v_H\ge4$ and $\delta_H\ge2$), which we now describe. 

\begin{defn}
For a graph $H$ with $v_H\ge 4$, we put  
\[
\lambda_*
=
\min\frac{e_H-e_{F}-1}{v_H-v_{F}}
\]
minimizing over all subgraphs $F\subset H$ with 
$2\le v_F<v_H$. 
\end{defn}

Note that $\lambda_*>0$ if and only if $\delta_H\ge 2$, 
and we continue to restrict our attention to graphs $H$ with this property.

It is easy to see that $\lambda_*\le\lambda$, with equality
if and only if $H$ is balanced (see \cref{L_lamstar}). 
In \cref{S_REA} below, we show that 
a witness graph $W_e$ for some $e\in E[\l G\r_H]$ 
with $k\ge v_H$ vertices has at least $\lambda_*(k-v_H)+e_H-1$ edges. 
Note that $K_r$ is balanced, and so this reduces to $\lambda(k-2)+1$
in that case, recovering Lemma 9 in \cite{BBM12}.
A general lower bound follows. 

\begin{thm}\label{T_LBgen}
For any graph $H$ with $v_H\ge 4$ and $\delta_H\ge 2$,
\[p_c(n,H)\ge\Omega(n^{-1/\lambda_*}(\log{n})^{1/\lambda_*-1}).\] 
\end{thm}

Note that, in the case that $H=K_4$,  this lower bound  
includes the correct poly-logarithmic factor (recall that $p_c(n,K_4)\sim1/\sqrt{3n\log n}$, 
as 
discussed above), since $K_4$ is balanced and 
so $\lambda_*=\lambda=2$. 
In \cite{BBM12}, the ``double-dumbbell' $H=DD_r$ (two copies of 
$K_r$, $r\ge4$, joined by a pair of disjoint edges) 
is given as an example of an (unbalanced) 
graph for which $p_c=n^{-1/\gamma+o(1)}$, with 
$\gamma\in(\lambda',\lambda)$ (recall $\lambda'$ defined above \cref{D_bal}). 
We note that, in this instance, $\gamma=[{r\choose2}+1]/r=\lambda_*$.

On the other hand, 
Bidgoli, Mohammadian and Tayfeh-Rezaie \cite{BMTR21} 
have studied the cases $H=K_{2,t}$. 
For these unbalanced graphs, even finding 
the correct power $\gamma$ in $p_c=n^{-1/\gamma+o(1)}$ remains open, 
except in the specific case $H=K_{2,4}$ where 
$p_c=\Theta(n^{-10/13})$. However, note that $\lambda_*=1$ for this graph.

Towards a full solution to Problem 1 in \cite{BBM12}, 
it would be interesting to determine the class of graphs for which 
$p_c=n^{-1/\lambda_*+o(1)}$. 
We note that \cref{C_Hmain} shows that this class includes all balanced graphs. 

\subsection{Strictly balanced graphs}

Finally, we turn our attention to the specific class of balanced graphs, which includes
the cases $H=K_r$, $r\ge5$, of natural interest. 

\begin{defn}\label{D_SB}
We call $H$ {\it strictly balanced}
if the inequality $(e_F-1)/(v_F-2)< \lambda$
is strict 
in \cref{D_bal}. 
\end{defn}

Note that $\lambda>1$ for strictly balanced graphs. 
Indeed, recall that we assume that $\delta_H\ge2$, and 
consider any 
$F\subset H$ with three vertices and at least
two edges. 

For this class of graphs, we prove a sharper upper bound.

\begin{thm}\label{T_Hmain}
If $H$ is strictly balanced then 
$p_c(n,H)\le O(n^{-1/\lambda})$. 
\end{thm}

Note that $K_4$ is balanced, but not strictly. For all $r\ge5$, however, 
the graphs $K_r$ are strictly balanced. 
It is somewhat tempting to suspect that $p_c(n,H)= \Theta(n^{-1/\lambda})$ 
for all strictly balanced graphs,
but we would not go so far as to make that conjecture.

\subsection{Outline}
The general lower bound in \cref{T_LBgen} is proved in \cref{L_LBgen},
using \cref{P_lamstar} below, which lower bounds the number of edges in a 
witness graph. This result generalizes a result about cliques $K_r$, proved 
by a different strategy in \cite[Lemma 9]{BBM12}, 
to all graphs $H$.

Finally,  in \cref{L_UBstrict}, we prove 
the upper bound in \cref{T_Hmain} for strictly balanced graphs.
This is the most technical part of the article, where 
two rounds of the second moment method are required.
First we bound the probability that a given edge is added by 
a specific type of witness graph, called an ``$H$-ladder,''
with an appropriately chosen height. Then we show that these
events are roughly independent enough to ensure that a large proportion of 
all edges are added in this way, and then full percolation follows  
by sprinkling. 
We note that $H$-ladder graphs were introduced in 
\cite[Section 2]{BBM12}.
A finer analysis of the ways in which pairs of such graphs can
overlap is key to the sharper upper bound in \cref{T_Hmain}.

\subsection{Notation}

For a graph $G=(V,E)$, we denote its vertex set by $V[G]=V$ and its edge set by $E[G]=E$, 
with sizes $v_G=|V[G]|$ and $e_G=|E[G]|$. We write $F\subset G$ for a 
(not necessarily induced) subgraph $F$ of $G$.
We denote by $G\setminus e$ the graph obtained from $G$ by keeping its vertex set, 
and deleting the edge $e$ from its edge set. 
We similarly use $G\setminus\{e_1,\ldots,e_k\}$ when the edges $e_1,\ldots,e_k$ are removed. 
For graphs $G_1, G_2,\ldots,G_k$, 
we denote by $\bigcup_{i=1}^k G_i$ the graph with vertex set $\bigcup_{i=1}^k V[G_i]$ 
and edge set $\bigcup_{i=1}^k E[G_i]$, whereas $\bigcap_{i=1}^k G_i$ 
is the graph with vertex set $\bigcap_{i=1}^k V[G_i]$ and edge set $\bigcap_{i=1}^k E[G_i]$.
For an edge $e=\{x,y\}$, for ease of notation, we often simply write
$e$ to denote the graph $G$
with $V[G]=\{x,y\}$ and $E[G]=\{e\}$.

Throughout the paper we use the notation $f=O(g)$ or $f\le O(g)$ 
for functions $f$ and $g$ if $f\le cg$ for some universal constant $c$ not depending on any 
of the arguments of $f$ and $g$. Similarly, we use $f=\Omega(g)$ or $f\ge \Omega(g)$ for the 
opposite inequality, and $f=\Theta(g)$ when both of these statements hold.

\subsection{Acknowledgements}
We are indebted to Yuval Peled for noticing 
an issue in a previous version of this work
that had claimed a lower bound (matching the upper bound in 
\cref{T_Hmain}) for all strictly balanced graphs. 

\section{A general lower bound}\label{L_LBgen}

In this section, we obtain a lower bound for $p_c$
that holds for all graphs $H$ (with $\delta_H\ge2$ and $v_H\ge4$). 
We first recall, in \cref{S_WSA,S_REA_BBM}, the results 
from \cite{BBM12} which we use. 
Then, in \cref{S_REA}, we prove \cref{T_LBgen}. 

\subsection{Witness set algorithm (WSA)}\label{S_WSA}
In \cite{BBM12} Section 3.1, 
the {\it witness set algorithm (WSA)}
is introduced, which 
assigns a {\it witness graph} $W_e\subset G$ to each $e\in E[\l G\r_H]$
such that $e\in E[\l W_e\r_H]$.
These graphs are defined in time with the percolation dynamics in the following way. 
Let $E_t$ denote the set of edges $E[G_t]\setminus E[G_{t-1}]$ 
added at time $t$.
For any edge $e\in E[G]$, $W_e$ consists of the single edge $e$ with its two endpoints as vertices.
Then, at step $t\ge1$, the WSA defines simultaneously for each $e\in E_t$ the witness graph
$W_e=\bigcup_{f\in E[H_e\setminus e]} W_f$, where 
$H_e$ is a copy (chosen arbitrarily if not unique)  
of $H$ completed by the addition of $e$ at time $t$. 
Since $H_e\setminus e\subset G_{t-1}$, this procedure
is well-defined. 

\begin{defn}
For a witness graph $W_e$, we define $v_{W_e}-2$, the number of vertices
in $W_e$ besides the endpoints of $e$, to be its {\it size}.
\end{defn}

We note that the ``size'' of a graph sometimes refers to its number of edges.
We will not follow this 
convention.

A key property of this construction is the 
following Aizenman--Lebowitz \cite{AL88} type property 
(cf.\ Lemma 13 in \cite{BBM12}), as is easily observed. 

\begin{lem}\label{L_HAL}
Suppose that  $W_e$ for some $e\in E[\l G\r_H]$
is of size at least $k$ for some $k\ge v_H-2$. Then, for some 
$k'\in[k,e_Hk]$, there is an $f\in E[\l G\r_H]$
so that $W_f$ is of size $k'$. 
\end{lem}

\begin{proof}
Let $M_t$
be the maximal size of a witness graph $W_f$, for $f\in E[G_t]$. 
Note that $M_0=0$ and $M_1=v_H-2$ (assuming $E_1\neq\emptyset$). 
Then for any $e\in E_{t+1}$, $t\ge1$,  
$W_e$ is of size at most
$v_H-2+(e_H-1)M_t\le e_H M_t$, since
$M_t\ge M_1 =v_H-2$. 
Therefore $M_{t+1}\le e_H M_t$. Hence, if $M_t\ge k$ for some $t$, then 
$M_s\in[k,e_Hk]$ for some $s\le t$. 
\end{proof}

\subsection{Red edge algorithm (REA)}\label{S_REA_BBM}

In \cite{BBM12} a lower bound for $p_c$, in the special case of $H=K_r$, 
is obtained using \cref{L_HAL} together with a lower bound for the number
of edges in a witness graph. Specifically, it is shown that if $W_e$
is of size $k$, then it has at least $\lambda k+1$ edges. 
A key tool in this regard is the  
{\it red edge algorithm (REA),}
which is based on WSA (see \cref{S_WSA} above).  
Informally, for a given edge $e\in E[\l G\r_H]\setminus E[G]$
(not in $G$ but eventually added by the $H$-percolation
dynamics), REA describes the construction of the witness graph 
$W_e$ one step at a time 
by running WSA, but ignoring steps that do not contribute
to the construction of $W_e$. All involved edges which are not in $G$
are colored red. 

REA is discussed in Section 3.1 of \cite{BBM12}. 
For completeness, we briefly describe the construction here.
First, we ``slow down'' the 
$H$-percolation dynamics, so that in each step a single new edge $e'$
is added. Recall (see \cref{S_WSA}) that  
$H_{e'}$ is the copy of $H$ which $e'$ completes (that is, $W_{e'}=\bigcup_{f\in E[H_{e'}\setminus e']} W_f$).  
Let 
$e_1,\ldots, e_m=e$ be the edges for which $W_{e_j}\subset W_e$ for $j\in\{1,\ldots,m\}$
that are added (in that order by the ``slowed down'' dynamics) 
until finally $e$ is added. Color all 
 $e_1,\ldots, e_m$ red. 
Then 
\begin{equation}\label{E_redE}
W_e=(H_1\cup\ldots\cup H_m)\setminus \{e_1,\ldots,e_m\}, 
\end{equation}
where $H_j=H_{e_j}$.  
Hence, REA has $m$ steps. In the $j$th step, a copy $H_j$ of $H$ is added
and one of its edges $e_j$ is colored red. Note that $e_j\notin \bigcup_{i<j} E[H_i]$, 
however, some of the other edges in $H_j$ may already 
be in 
$\bigcup_{i<j} E[H_i]$.

\subsection{Proof of the lower bound}\label{S_REA}

In this section, we give a proof of \cref{T_LBgen} that is based on the 
Aizenman--Lebowitz type property of \cref{L_HAL}, 
and the following lower bound on the number of edges of witness graphs, 
which we will prove later (as a special case of 
\cref{L_HREA} below). Recall the definition of $\lambda_*$ given in \cref{S_intro}.

\begin{prop}\label{P_lamstar}
If $W_e$ is a witness graph for an edge  $e\in E[\l G\r_H]$ on $k\ge v_H$
vertices, then 
$W_e$ has at least $\lambda_*(k-v_H)+e_H-1$ edges. 
\end{prop}

The next lemma shows, in particular, that 
\[
\lambda_*(k-v_H)+e_H-1\ge \lambda_*(k-2)+1
\]
with equality if $H$ is balanced. 

\begin{lem}\label{L_lamstar}
We have $\lambda_*\le \lambda$, with equality if and only
if $H$ is balanced. 
\end{lem}

\begin{proof}
The case when $F$ consists of a single edge $e\in E[H]$, 
with its two endpoints as vertices, shows that $\lambda_*\le \lambda$. 
To see the second claim, note that, for $F\subset H$ with $3\le v_F<v_H$, 
\[
\lambda = \frac{e_H-\lambda(v_F-2)-2}{v_H-v_F}\le \frac{e_H-e_F-1}{v_H-v_F}
\]
if and only if $\lambda(v_F-2)\ge e_F-1$.
\end{proof}

Since $K_r$ is balanced,  and so $\lambda_*=\lambda$ in this case, 
we obtain Lemma 9 in \cite{BBM12} as a special case
of \cref{P_lamstar}.
We note here that, as per \cite{YuvalPeled}, this result, in the 
special case of $K_r$, 
can also be alternatively obtained using \cite{Kal84}. 

We note here that once \cref{L_HAL,P_lamstar} have been established, 
the proof of our 
general lower bound  for $p_c$
follows straightforwardly, along the same lines as the proof of Proposition 8 in Section 3.2 of \cite{BBM12}. 

\begin{proof}[Proof of \cref{T_LBgen}]
Fix $e\in E[K_n]$. 
Let $p=\alpha n^{-1/\lambda_*}(\log{n})^{1/\lambda_*-1}$. 
We show that, for $\alpha>0$ sufficiently small, 
$e\notin E[\l\Gnp\r_H]$ with high probability.

If $e\in E[\l\Gnp\r_H]$, then 
by \cref{L_HAL}
either (1) $e\in E[\Gnp]$, (2) 
$W_e$ is of size $k\in[v_H-2,\log n]$, or else,  
(3) some $W_f$ is of size 
$k'\in(\log{n},e_H\log n]$. 
By \cref{P_lamstar}, and the remark after it,  
a witness graph of size $k$ has at least $\lambda_*k+1$
edges. 
There are at most
$k[O(k^{\lambda_*})]^k$  graphs with exactly $\lambda_*k+1$ 
edges on a given set of $k+2$ vertices (using ${m\choose \ell}\le (m e/\ell)^\ell$). 
Therefore, taking a union bound, 
\begin{multline*}
\P(e\in E[\l\Gnp\r_H])\\
\le p
+p\sum_{k=v_H-2}^{\log n} k[O(np^{\lambda_*} k^{\lambda_*-1})]^k
+n^2p\sum_{k'=\log n}^{e_H\log n} k'[O(np^{\lambda_*} (k')^{\lambda_*-1})]^{k'}\\
\le p[O(\log n)]^2(1+n^{2+O(\log(\alpha))})\ll1
\end{multline*}
for $\alpha$ sufficiently small. 
\end{proof}

We turn now to the proof of \cref{P_lamstar}. 
To this end, it is useful to define an increasing 
sequence 
$\cH_1\subset \cdots\subset\cH_m$ of 
auxiliary 
graphs
associated with REA.
(We note that something similar appears in 
\cite{BBM12}, however, using hyper-graphs. 
For our purposes, it suffices to use
graphs.) 
Recall (see \cref{S_REA_BBM}) that in the $j$th step of REA
a copy $H_j$ of $H$ is added and one of its new
(not in $\bigcup_{i<j} H_i$) edges $e_j$
is colored red. 
Let $\cH_0$ be the empty graph. Then the graph $\cH_j$ is obtained from 
$\cH_{j-1}$
by adding a new vertex
$v_j$, which we associate with $H_j$. 
For each edge $e'$ in $E[H_j]\cap E[\bigcup_{i<j} H_i]$ we 
add an edge from $v_j$ to $v_k$, where $k=\max\{i<j:e'\in H_i\}$. 
Finally, delete any redundant edges, to ensure that $\cH_j$ is a simple graph. 
Note that $v_k$ is associated with the most recently 
(before step $j$)
added copy of $H$ containing $e'$. 
Finally, we put $\cH_e=\cH_m$. 

\begin{defn}\label{D_COMP}
For each connected component
$\cC$ of $\cH_j$, we refer to 
$C=\bigcup_{v_i\in V[\cC]} H_i$
as the corresponding {\it component} in 
$\bigcup_{i=1}^j H_i$. 
\end{defn}

Components of 
$\bigcup_{i=1}^j H_i$ can share vertices but not edges. 
Indeed, when $H_j$ is added in step $j$ of REA, 
all components in $\bigcup_{i=1}^{j-1} H_i$ 
with at least one edge in $H_j$ are merged with $H_j$ to obtain a component of 
$\bigcup_{i=1}^{j} H_i$.  
Also note that  
$\bigcup_{i=1}^m H_i$ has only one component (that is, 
$\cH_e$ is connected). 
To see this, simply recall that  
$W_e=\bigcup_{f\in E[H_e\setminus e]} W_f$, and so note that each   
$W_f$ has an edge $f\in H_e$, by the construction of $W_e$. 
Therefore we obtain \cref{P_lamstar} by the next result
(which also implies Lemma~10 in \cite{BBM12} as a special case, 
once the terminology there is unpacked). 

\begin{lemma} \label{L_HREA}
Let $W_e$ be a witness graph for an edge  
$e\in E[\l G\r_H]\setminus E[G]$ on $k\ge v_H$
vertices. 
Then, after any $j\ge 1$ number of steps of the corresponding instance of REA, 
any component $C$ of $\bigcup_{i=1}^j H_i$   
has at least $\lambda_*(v_C-v_H)+e_H-1$ non-red edges. 
\end{lemma}

\begin{defn}\label{D_lamstar}
We let $\cV_*\subset\{2,\ldots,v_H-1\}$ be the set 
for which $\lambda_*$ is attained by some subgraphs $F\subset H$ with 
$v_F\in\cV_*$. We put 
\[
\xi=
\min \frac{e_H-e_{F}-1}{v_H-v_{F}}
-\lambda_*
\]
minimizing over $F\subset H$ with 
$v_F\in\{2,\ldots,v_H-1\}\setminus \cV_*$.  
\end{defn}

Roughly speaking, we shall see that the most efficient way to add a new copy $H_j$
of $H$,  
in any given step $j$ of REA, 
is to ensure that the number of vertices 
in $V(H_j)\cap V(C')$
is in $\cV_*$, for each of the components $C'$ in
$\bigcup_{i=1}^j H_i$ that have an 
edge in common with $H_j$ (except in the simple case 
$H_j\setminus e_j\subset C'$, when only the red edge $e_j$
is added to $C'$ to obtain $C$). 
Note that all such components are merged with $H_j$ 
in the $j$th step of REA to form a new component $C$. 
The quantity $\xi$ is related to 
the minimum ``cost'' of a non-optimal 
merge in REA.

We will also use the following distinction of steps in REA.

\begin{defn}\label{D_cases}
We call the $j$th step in REA a {\it type-1 step} if $H_j\setminus e_j$ 
is contained in some component of $\bigcup_{i=1}^{j-1} H_i$ as a subgraph. 
Otherwise, we call it a {\it type-2 step}. 
In particular, we include steps in which a new component is formed as a trivial instance of a type-2 step. 
\end{defn}

\begin{proof}[Proof of \cref{L_HREA}]
The proof is by induction on the number of steps $j$ taken. 
The base case $j=1$ is trivial, since  
$H_1\setminus e_1$ has $v_H$ vertices and $e_H-1$ edges.  
Likewise, the same reasoning applies if in some step $j>1$ a new
component is created. 
Hence suppose that in step $j>1$, the addition of $H_j$ causes
exactly $h\ge1$ (edge-disjoint) components $C_1,\ldots, C_h$
(each with at least one edge in $H_j$) to merge with $H_j$ 
into a single component $C$. 
By assumption, we assume that the $C_i$ 
have $k_i$ vertices and 
$\lambda_* (k_i-v_H)+e_H-1+\ell_i$ non-red edges, for some $\ell_i\ge0$. 
Let $k$ denote the number of vertices in $C$.
Note that 
\[
k=v_H+\sum_i(
k_i-\eps_i-\eta_i)
\]
where  $\eps_i$ is the number of vertices in $V[C_i]\cap V[H_j]$, and 
$\eta_i$ is the number of vertices in 
$(V[C_i]\setminus V[H_j])\cap \bigcup_{i'<i}V[C_{i'}]$
(that is, other vertices in $C_i$ that are in a previously 
considered $C_{i'}$).
To complete the proof, we show that $C$ has at least 
$\lambda_*(k-v_H)+e_H-1$ non-red edges. We distinguish two scenarios based on whether the $j$th step is a 
type-1 or type-2 step.

{\it Case 1.} If $H_j\setminus e_j$ is contained in some component of $\bigcup_{i=1}^{j-1} H_i$ as a subgraph, 
then necessarily $h=1$, since the components $C_i$ are edge-disjoint.  
In this case, the result follows immediately, 
since then $k=k_1$ and a single red edge (and no non-red edges)
is added to form $C$. 

{\it Case 2.} On the other hand, suppose that no $C_i$ contains 
$H_j\setminus e_j$ as a subgraph. 
It is more convenient to start with $H_j$ and color $e_j$ red, 
and then merge the $C_i$ with it one at a time. 
In these dynamics, any edge in 
$C_i\cap H_j$ that is red in $C_i$ remains red after merging. 
Initially, we have the $e_H-1$ non-red edges in 
$H_j$. In the $i$th sub-step (when we merge $C_i$),
$k_i-\eps_i-\eta_i$ vertices are added. 
Note that, by the choice of $\lambda_*$ and $\xi$, 
at least 
\[
(v_H-\eps_i)(\lambda_*+\xi\1_{\eps_i\notin \cV_*})+\1_{\eps_i=v_H}
\] 
of the 
$e_H-1$ 
edges in $H_j\setminus e_j$
are not in $C_i$.
Therefore, since the components $\{C_{i'}\}_{i'\le i}$ are edge-disjoint, 
the number of non-red edges increases by at least 
\begin{equation}\label{E_REA}
\lambda_*(k_i-\eps_i)+\ell_i
+(v_H-\eps_i)\xi\1_{\eps_i\notin \cV_*}+\1_{\eps_i=v_H}
\end{equation}
in the $i$th sub-step.
Altogether, summing over all $i$, we find that  $C$ has $k$ vertices and at least
\[
(e_H-1)+\lambda_*\sum_i(k_i-\eps_i)+\sum_i[\ell_i
+(v_H-\eps_i)\xi\1_{\eps_i\notin \cV_*}+\1_{\eps_i=v_H}]
\]
non-red edges. Finally, note that the above is equal to 
\begin{multline}\label{E_excess}
(e_H-1)+\lambda_*(k-v_H+\sum_i\eta_i)
+\sum_i[\ell_i
+(v_H-\eps_i)\xi\1_{\eps_i\notin \cV_*}+\1_{\eps_i=v_H}]\\
=[\lambda_*(k-v_H)+e_H-1]
+\sum_i[\ell_i+\lambda_*\eta_i
+(v_H-\eps_i)\xi\1_{\eps_i\notin \cV_*}+\1_{\eps_i=v_H}]\\
\ge \lambda_*(k-v_H)+e_H-1,
\end{multline}
as required. 
\end{proof}

\section{Upper bound for strictly balanced $H$}\label{L_UBstrict}

Next, we prove the upper bound
in \cref{T_Hmain}. 
We show that, for a strictly balanced graph $H$,  
with high probability $\l\Gnp\r_H=K_n$, if $p=(\alpha/n)^{1/\lambda}$
and $\alpha>0$ is sufficiently large.

Two applications of the second moment method are involved. 
First we show that with probability bounded away from 0 
(and tending to 1 as $\alpha\to\infty$) any given edge
$e\in E[K_n]$ is added in $\l\Gnp\r_H$ due to a simple
type of witness graph, called an {\it $H$-ladder}. 
These graphs were considered in \cite{BBM12}. 
The main difference here is that we consider {\it induced}
$H$-ladders, 
resulting in an easier analysis of 
correlations (overlapping 
ladders). Then we show that the events that two given edges are added
in $\l\Gnp\r_H$ 
by induced $H$-ladders (of suitable heights) are roughly independent. 
Hence, a significant proportion 
(tending to 1 as $\alpha\to\infty$)
of all ${n\choose2}$ edges
in $K_n$ are included in $\l\Gnp\r_H$. 
Full percolation is then easily deduced 
(by Tur\'an's Theorem and sprinkling).

\subsection{$H$-ladders}\label{S_LadderLemmas}

We consider 
(as in Section 2 of \cite{BBM12})
the following type of 
edge-minimal witness graph, 
where the associated graph (as in the discussion before \cref{D_COMP})
is a path. 

\begin{defn}\label{D_ladder}
An {\it $H$-ladder} $L$ of height $h$
(see \cref{F_ladder} below)
is a graph 
constructed 
using $h$ copies 
$S_i$ (called {\it steps})
of $H$ minus two non-incident edges (called {\it rungs}) 
$\{u_{i-1},v_{i-1}\}$ and 
$\{u_{i},v_{i}\}$ such that, for each $1<i\le h$, we have that 
$V(S_i)\cap \bigcup_{j<i}V(S_j)=\{u_{i-1},v_{i-1}\}$. 
We then obtain $L$ as the union of the $S_i$ and the {\it top rung} 
$\{u_h,v_h\}$. 
We call 
$(v_H-2)h$ the {\it size,}
$h$ the {\it height} and 
$\{u_0,v_0\}$ the {\it base} of $L$. 
We often write $k=(v_H-2)h$.
\end{defn}

By induction on $h$, it is easy to see that $L$ is a witness graph for its base edge. Note that $L$ has 
$\lambda k+1$ edges, the minimal possible number by \cref{P_lamstar}. To see this note that each $S_i$ has 
$e_H-2=\lambda(v_H-2)$ edges. Since only the top rung 
$\{u_h,v_h\}$ is included in $L$, it has only 
$\lambda(v_H-2)h+1=\lambda k+1$ edges in total. 

It is convenient, although slightly informal, to speak of vertices and edges
of a ladder $L$ that are ``above'' and ``below'' its various rungs, etc. 
In this sense, note that $\lambda$ is the average number of edges
``sent down the ladder'' by vertices ``above'' the base. 

\begin{figure}[h]
\centering
\includegraphics[scale=0.75]{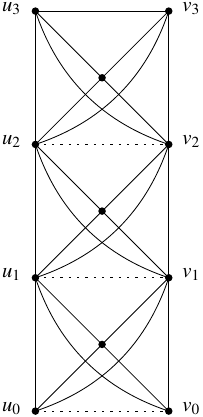}
\caption{
A $K_5$-ladder of height $h=3$ and size $k=(5-2)h=9$
has $k+2=11$ vertices and $\lambda k+1=[{5\choose2}-2]h+1=25$ edges.  
}
\label{F_ladder}
\end{figure}

Let us note here that we will, beginning in the next section, restrict to a specific class of $H$-ladders
(see \cref{D_uniformL} below) that is simpler and suffices for our purposes. 
However, before doing so, we first establish the following result that holds
for $H$-ladders in general. 

In \cite{BBM12} (see Lemma 6) it is shown that any subgraph $X\subset L$ containing 
$x+2< k+2$ vertices of $L$, including those in its base, has at most $\lambda x$ edges. 
Equality is obtained if 
$X=\bigcup_{i\le h'} S_i$ for some $1\le h'<h$. 
We prove the following estimate, 
which bounds the 
inefficiency of edge sharing in the other cases. 

Since $H$ is strictly balanced (see \cref{D_lamstar}), 
\[
\xi=
\min\frac{e_H-e_{F}-1}{v_H-v_{F}}
-\lambda>0
\]
minimizing over $F\subset H$ with $3\le v_F<v_H$. 
The case of $F$ with $v_F=v_H-1$ gives the bound 
$\xi\le \delta_H-1-\lambda\le \delta_H-2$ 
(recall that $\delta_H\ge2$, and so 
$\lambda\ge1$). 

\begin{lem}\label{L_Hladder}
Let $L$ be an $H$-ladder of size $k=(v_H-2)h$. 
Let $X$ be a proper induced subgraph of $L$
that contains $x$ vertices above the base of $L$. 
Then $X$ has at most $\lambda x-\xi \sigma$ edges, where
$\sigma$ is the number of steps $S_i\not\subset X$ of $L$ such that 
$X$ contains at least one vertex
in $V[S_i]\setminus\{u_{i-1},v_{i-1}\}$. 
\end{lem}

Note that $\sigma=0$ if and only if 
$x=0$ or 
$X=\bigcup_{i\le h'} S_i$ for some $1\le h'<h$. 
This result, in particular, implies Lemma 6 in \cite{BBM12}
(without the condition $\lambda\ge2$). 
Also note that we do not require that $X$ contains the base vertices of $L$. 
This allows for an easier inductive proof, and 
will be useful for analyzing overlapping ladders with different bases
(\cref{L_2ndMM} below). 

\begin{proof}
The proof is by induction on the height $h$ of $L$. 
If $x=0$ (and so also $\sigma=0$) the statement is trivial. 
Thus we assume $x\ge1$. 

{\it Base case.} 
If $h=1$ then $1\le x\le v_H-2$ and $\sigma=1$ (since $X$ is proper). 
There are $\lambda(v_H-2)+1$ edges in $L$. 

{\it Case 1a.}
If $x=v_H-2$, then 
at least one vertex in the base of $L$ is 
not in $X$.
Hence there are at least $\delta_H-1\ge \xi+1$ edges in $E[L]\setminus E[X]$,
and so at   
most $\lambda x-\xi$ in $E[X]$. 

{\it Case 1b.}
If  $1\le x< v_H-2$, then there are at least 
$(\lambda+\xi)(v_H-2-x)+1$
edges in $E[L]\setminus E[X]$,
and so at most  
$\lambda x-\xi(v_H-2-x)\le \lambda x-\xi$
in $E[X]$. 

{\it Inductive step.} Suppose  $h>1$.   
Let $L'\subset L$ be the ladder of height $h-1$ 
based at the first rung $\{u_1,v_1\}$ of $L$. 

{\it Case 2.} If $S_1\subset X$ or $V[X\cap S_1]\subset\{u_0,v_0\}$, 
then the result follows immediately by the inductive
hypothesis applied to $L'$ (since in either case $X\cap L'\subset L'$ is 
proper). 

{\it Case 3.} Suppose that $X$ contains $x_1\ge1$
vertices in 
$S_1\setminus\{u_0,v_0\}$ and $S_1\not\subset X$. 
Then, by the base case,  
$X$ contains at most $\lambda x_1-\xi-\1_{u_1,v_1\in V[X]}$ edges in $S_1$
(since $h>1$, the edge $\{u_1,v_1\}\notin E[L]$). 

{\it Case 3a.}
If $X\cap L'= L'$ (in which case $\sigma=1$, 
and $u_1,v_1\in V[X]$)
the claim follows, since then there are $\lambda x_1-\xi-1$ 
edges in $X$ below the first rung, $\lambda(k-x_1)+1$
above, and so $\lambda x-\xi$ in total.

{\it Case 3b.}
Otherwise, applying the inductive hypothesis
to the remaining $x-x_1$ vertices of $X$ in $L'$, 
it follows that there are at most $\lambda(x-x_1)-\xi(\sigma-1)$
edges in $X\cap L'$. Hence $L$ has at most 
$\lambda x-\xi\sigma$ edges.  
\end{proof}

\subsection{$H$-Ladders in $\Gnp$}

Having established \cref{L_Hladder}, 
we turn to the upper bound for $p_c$. 
We first obtain a lower bound on the probability that a given 
edge $e\in E[K_n]$ is the base of an $H$-ladder of height $h$ in 
$\Gnp$. This gives a lower bound on the probability that $e\in E[\l \Gnp\r_H]$. 
We then verify the approximate independence for different bases. 
This strategy thus involves 
two applications of 
the second moment method. 
As already discussed, 
we restrict to the case of 
induced $H$-ladders, since
this simplifies the analysis of correlations. 
Furthermore, we also restrict our attention to a 
specific type of $H$-ladder, defined as follows. 

\begin{defn}
Fix two non-incident edges $e_t,e_b$ in $H$, 
and a copy $T$ of $H\setminus\{e_t,e_b\}$
labelled in some arbitrary way such that 
\begin{itemize}[nosep]
\item the vertices in $e_b$
are labelled by $\{1,2\}$,  and 
\item all other vertices (not in $e_b$) in $H$
are labelled by $\{3,4\ldots, v_H\}$. 
\end{itemize} 
We call $T$ the {\it template}. 
\end{defn}

We fix $T$ for the rest of this work. 
We use the template $T$ to define a simple class
of $H$-ladders, where, informally speaking, 
copies of $T$ are stacked on top of each other. 

\begin{defn}\label{D_uniformL}
A (labelled) $H$-ladder 
(see \cref{D_ladder}) 
is {\it uniform} if, for each of its steps $S_i$, the function $\phi_i$ for which
\begin{itemize}[nosep]
\item $\phi_i(u_{i-1})=1$,  
$\phi_i(v_{i-1})=2$, and 
\item $\phi_i(w_k)=k+2$, for 
$w_k\in V[S_i]\setminus\{u_{i-1},v_{i-1}\}$ 
of $k$th largest label, 
\end{itemize} 
is an isomorphism from $S_i$ to the template $T$. 
\end{defn}

Note that, since $T$ is fixed,  there are exactly   
\[
{k\choose v_H-2,\ldots,v_H-2}=\frac{k!}{(v_H-2)!^h}
\]
uniform $H$-ladders of size $k=(v_H-2)h$
(and height $h$)
on a given set of $k+2$ vertices and with a given base $e$. 
Indeed, the conditions in \cref{D_uniformL}
imply that
the only freedom in selecting such an $H$-ladder
is in choosing which  
vertices are in each of the $h$ sets
$V[S_i]\setminus\{u_{i-1},v_{i-1}\}$. 

\begin{defn}
For $\eps\in(0,1)$, put
\begin{equation}\label{E_alpha_eps}
\alpha_\eps=\exp\left[
\left(\frac{v_H+1}{\xi}+\frac{1-\eps}{4}\right)\frac{\log(v_H-2)}{\eps}
\right]
\end{equation}
and
\begin{equation}\label{E_beta_eps}
\beta_\eps=\frac{\eps\xi}{\lambda(v_H+1)(v_H-2)\log(v_H-2)}.
\end{equation}
\end{defn}
Note that $\alpha_\eps\to\infty$ and 
$\beta_\eps\to0$, as $\eps\to0$. 
We also note that $\alpha_\eps,\beta_\eps$
are chosen so that \eqref{E_ENk>>1},
\eqref{E_beta_eps} and \eqref{E_deltaeps2} below simplify nicely 
(but are, of course, not the only choice of $\alpha,\beta$ for which 
the following lemmas hold). 

First we show that 
a given edge in $K_n$ is the base of an $H$-ladder  
with probability tending to 1 as 
$np^\lambda\to\infty$. 

\begin{lem}\label{L_gamma}
Fix $\eps\in(0,1)$.
Put $np^\lambda=\alpha_\eps(v_H-2)!^{1/(v_H-2)}$. 
Then any given $e\in E[K_n]$ is the base of an induced 
uniform $H$-ladder of height $h=\beta_\eps\log n$ in $\Gnp$ with probability
at least $\gamma_\eps-o(1)$, where
\[
\gamma_\eps=1-\frac{1}{\alpha_\eps^{v_H-2}-1}.
\]
\end{lem}

For ease of exposition, we write quantities such as $h=\beta_\eps\log n$ 
as is, instead of replacing them with their integer parts.  

\begin{proof}
Let $N_k$ denote the 
number of induced uniform $H$-ladders 
in $\Gnp$ 
of size $k=(v_H-2)h$ with a given base $e$. 
Noting that $k\ll n$ and $k^2p\ll1$ (and using the
standard bound ${n\choose k}\ge(n-k)^k/k!$) we find that 
\begin{equation}\label{E_ENk}
\E N_k =
{n-2\choose k} 
\frac{k!}{(v_H-2)!^h}
p^{\lambda k+1}(1-p)^{{k+2\choose2}-(\lambda k+1)}
\ge p\alpha_\eps^k(1-o(1)).
\end{equation}
Since, by \eqref{E_alpha_eps}
and \eqref{E_beta_eps},  
\begin{equation}\label{E_ENk>>1}
\lambda\beta_\eps(v_H-2)\log\alpha_\eps
=1+\frac{(1-\eps)\xi}{4(v_H+1)}>1,
\end{equation}
it follows that $\E N_k\gg1$.  

Let $L_1,L_2,\ldots$ enumerate all uniform $H$-ladders 
in $K_n$ of size $k$
that are based at $e$. Let $A_i$ be the event that $L_i$ is an induced
subgraph of $\Gnp$. 
Following Section 4.3 of \cite{AS16}, and using symmetry, 
\begin{align}\label{EQ_2ndMM}
\P(N_k=0)
&\le\frac{\E(N_k^2)}{(\E N_k)^2}-1
=\frac{1}{\E N_k}+\frac{1}{(\E N_k)^2}\sum_{i\neq j}\P(A_i\cap A_j)-1\\\nonumber
&=\frac{1}{\E N_k}\left[1+\sum_{i>1}\P(A_i|A_1)\right]-1. 
\end{align}
Since we are considering
induced subgraphs of $\Gnp$, for any $i\neq 1$, 
we have $\P(A_i|A_1)=0$ unless $V[L_i]\neq V[L_1]$. 
Let $S_i$, $1\le i\le h$, denote the steps of $L_1$. 

{\it Case 1.} 
First, we consider the case that $L_i$ ``breaks cleanly'' from $L_1$ at one of its rungs, that is, $L_i\cap L_1=e$ 
or $L_i\cap L_1=\bigcup_{i\le h'} S_i$, for some $1\le h'<h$.
 
If 
$L_i\cap L_1=e$ then 
$L_i$ is an $H$-ladder of height $h$ that is edge-disjoint
from $L_1$. Hence 
$\P(A_i|A_1)\le p^{\lambda k+1}$.
Similarly, if 
$L_i\cap L_1=\bigcup_{i\le h'} S_i$, for some $1\le h'<h$, 
then $L_i$ and $L_1$ agree up to height $h'$. The part
of $L_i$ that is ``above'' the intersection $L_i\cap L_1$ 
is an $H$-ladder 
(based at the $h'$th rung $\{u_{h'},v_{h'}\}$ of $L_1$) 
of height $h-h'$ 
that is edge-disjoint from $L_1$. 
Hence, in this case, $\P(A_i|A_1)\le p^{\lambda k'+1}$, where $k'=(v_H-2)(h-h')$.
Summing over all such $L_i$, using \eqref{E_ENk}, 
\begin{align*}
\frac{1}{\E N_k}\sum \P(A_i|A_1)
&\le 
\frac{1}{\E N_k}\sum_{h'=0}^{h-1}
\frac{n^{k'}p^{\lambda k'+1}}{(v_H-2)!^{h-h'}}\\
&\le 
(1+o(1))\sum_{h'=0}^{h-1} \alpha_\eps^{-(v_H-2)h'}\\
&\le \frac{1+o(1)}{1-1/\alpha_\eps^{v_H-2}}.  
\end{align*}

{\it Case 2.} 
Next, we show that all other cases are of lower order. 
If $L_i$ does not ``break cleanly'' (as in Case 1) 
from $L_1$ then 
by \cref{L_Hladder}, $\P(A_i|A_1)\le p^{\lambda(k-x)+1+\xi\sigma}$,
where $x$ is the number of vertices in $X=L_i\cap L_1$ above the base of $L_1$, and 
$\sigma\ge1$ is the number of $S_i\not\subset X$ such that 
$V[X]\cap (V[S_i]\setminus\{u_{i-1},v_{i-1}\})\neq\emptyset$. 

For any such $L_i$, let $s\ge0$ be the number of maximal subgraphs
$\bigcup_{i=h_1}^{h_2} S_i\subset X$, $h_1\le h_2$. 
For convenience, we refer to these as the  
sub-ladders that $L_1$ and $L_i$ have in common. 
Let $y\ge0$ denote the number of 
other vertices in $X$ (not inside a common sub-ladder). 
Note that $s+y\ge1$ since $\sigma\ge1$. 
We claim that 
\begin{equation}\label{E_sig}
\sigma\ge\max\{1,(s-1+y)/v_H\}\ge(s+y)/(v_H+1).
\end{equation}
To see this, 
note that there are at least $2(s-1)\1_{s\ge1}+y$ vertices
of $X$ in steps $S_i\not\subset X$,  
since 
if $\bigcup_{i=h_1}^{h_2} S_i\subset X$ is maximal and $h_1>1$, then 
$S_{h_1-1}\not\subset X$ and $u_{h_1-1},v_{h_1-1}\in V[X]$. 

Next, we claim that, for given $x,s,y$,  there are at most
\[
\binom{h+1}{2s}s!\binom{h}{s}2^s\binom{k}{y}\binom{n}{k-x}(k-x+y)!
\le [O(k^3)]^{s+y} n^{k-x}
\]
ladders $L_i$ to consider. To see this, observe that: 
\begin{enumerate}[nosep]
\item 
there are $\tbinom{h+1}{2s}$ 
ways to select the $s$
common sub-ladders (since this corresponds to 
selecting $2s$ rungs in $L_1$) and $s!$ ways to 
choose the order in which they can appear 
in $L_i$, 
\item   
there are at most $\tbinom{h}{s}$ 
possibilities for where the sub-ladders are located in $L_i$  
(choose a height
for the top rung of each), and $2^s$ ways to decided 
whether the top rung (as it appears in $L_1$) of each sub-ladder 
is the top or bottom 
rung as it appears in $L_i$, and 
\item the final three factors bound the choices for the 
$k-x+y$ 
other vertices in $L_i$ and 
their locations in $L_i$. 
\end{enumerate}
Note that in (2) we need not also consider the possibility 
that the labels in a top or bottom rung of a common sub-ladder are 
reversed (with respect to $e$) in $L_i$ as compared with how they appear
in $L_1$. Indeed, doing so would either produce 
a non-uniform $H$-ladder, or else one that is 
equivalent to $L_1$.

Hence, summing over all such $L_i$ with given $x,s,y$, we
find (using $\alpha_\eps>1$, $x\ge1$, $s+y\ge1$, \eqref{E_ENk} and \eqref{E_sig}) that 
\begin{align*}
\frac{1}{\E N_k}\sum \P(A_i|A_1)
&\le 
\frac{1}{\E N_k}[O(k^3)]^{s+y} n^{k-x}
p^{\lambda(k-x)+1+\xi\sigma}\\
&\le 
O[(\log n)^3\frac{(np^\lambda)^{k-x}}{\E N_k}
p^{1+\xi/(v_H+1)}]\\
&\le 
O[(\log n)^3(v_H-2)!^{h}
p^{\xi/(v_H+1)}]\ll n^{-\vartheta_\eps}, 
\end{align*}
where, by \eqref{E_beta_eps},  
\begin{equation}\label{E_delta}
\vartheta_\eps=
\frac{\xi}{\lambda(v_H+1)}-
\beta_\eps(v_H-2)\log(v_H-2)
=\frac{(1-\eps)\xi}{\lambda(v_H+1)}>0.
\end{equation}
Since there are only $O(k^3)$ relevant $x,s,y$ the same
holds summing over all $L_i$ (not included in Case 1). 

Therefore, combining the two cases, we find that 
\[
\frac{1}{\E N_k}\sum_{i\ge1}\P(A_i|A_1)
\le \frac{1+o(1)}{1-1/\alpha_\eps^{v_H-2}},
\]
and so, by 
\eqref{EQ_2ndMM}, 
\[
\P(N_k>0)\ge1-\frac{1+o(1)}{\alpha_\eps^{v_H-2}-1}
=\gamma_\eps-o(1).\qedhere
\]
\end{proof}

Next, by another application of the second moment method, 
we show that with high probability a significant proportion
(tending to 1 as $\alpha\to\infty$) 
of edges in $K_n$ 
are bases of $H$-ladders in $\Gnp$, and so included
in $\l\Gnp\r_H$. 

\begin{lem}\label{L_2ndMM}
Fix $\eps\in(0,1)$.
Put $np^\lambda=\alpha_\eps(v_H-2)!^{1/(v_H-2)}$. 
Then, with high probability, there 
are at least $(\gamma_\eps-\eps){n\choose2}$
edges in $K_n$ which are bases of induced
uniform 
$H$-ladders of height $h=\beta_\eps\log n$ in $\Gnp$. 
\end{lem}

\begin{proof}
Let $e_1,e_2,\ldots$ enumerate the edges of $K_n$
and let $E_i$ denote the event that $e_i$ is the base of an induced 
uniform $H$-ladder
of height $h$ in $\Gnp$. 
Similarly, let $E_i'$ denote the event that $e_i$ is the base of a (not necessarily induced) 
uniform $H$-ladder
of height $h$ in $\Gnp$. 
We show that 
\[
\frac{1}{[{n\choose2}\P(E_1)]^2}\sum_{i\neq j}\P(E_i\cap E_j)\le 1+o(1) 
\]
from which, together with \cref{L_gamma}, the result follows (as then 
the number of such edges $\sum_i\1_{E_i}$, 
divided by its expectation ${n\choose2}\P(E_1)$, 
converges to 1 in probability, see again 
Section 4.3 of \cite{AS16} and the technique used in 
\eqref{EQ_2ndMM}). 

To this end, we bound the event $E_i\cap E_j$ by the union of events (1)
$E_i'\circ E_j'$ that there are edge-disjoint (not necessarily induced) 
uniform $H$-ladders of heights $h$
based at $e_i$ and $e_j$, and, (2) $E_{ij}$ that there is an induced uniform $H$-ladder of height
$h$ based at $e_j$ that includes an edge of such a ladder based at $e_i$. 

By the van den Berg--Kesten (BK) inequality \cite{BK85} (the events $E_i'$ are increasing) and symmetry, 
\[
\P(E_i'\circ E_j')\le \P(E_1')^2= (1+o(1))\P(E_1)^2,
\]
where the last equality follows since $k^2p\ll1$. 
It thus suffices to show that 
\[
\sum_{i\neq j}\P(E_{ij})\ll n^4, 
\]
since by \cref{L_gamma}, the probability $\P(E_1)\ge\gamma_\eps-o(1)$ (and so, in particular, bounded 
away from 0
as $n\to\infty$). 

Let $L_1$ be a fixed uniform $H$-ladder of height $h$ in $K_n$ based at $e_1$, 
and  let $A_1$ denote the event that $L_1$ is an induced subgraph of $\Gnp$. 
For $j>1$, let $B_{j}$ be the event that there is an induced 
uniform $H$-ladder in $\Gnp$
of height $h$ based at $e_j$ that includes at least one edge in $L_1$. 
As in the previous proof, let $N_k$ denote the 
number of induced uniform $H$-ladders 
in $\Gnp$ 
of size $k=(v_H-2)h$ with a given base $e$. 
Note that $\E(N_k)\le p\alpha_\eps^k$. Hence, by symmetry, we have 
\begin{align*}
\sum_{i\neq j}\P(E_{ij})
&\le n^2\sum_{j>1}\P(E_{1j})\\
&\le n^{2}\E(N_k)\sum_{j>1}\P(B_j|A_1)\\
&\le n^2 p\alpha_\eps^k\sum_{j>1}\P(B_j|A_1). 
\end{align*}
Hence, it suffices to show that 
\[
p\alpha_\eps^k\sum_{j>1}\P(B_j|A_1)\ll n^2. 
\]

Finally, we estimate $\sum_{j>1}\P(B_j|A_1)$ by 
a union bound, considering the expected 
(conditioned on $A_1$)
number 
of induced $H$-ladders $L$ of 
height $h$ based at some $e\neq e_1$ that include at least one edge of $L_1$,
and hence at least one vertex not in its base $e_1$. 

At this point, the argument is similar to the proof 
of \cref{L_gamma}, and so we only sketch the details. 
As before, we take two cases with respect to whether
 $L$ and $L_1$ ``intersect cleanly''
(that is, if  $L\cap L_1=\bigcup_{i\le h'}S_i$
for some $1\le h'<h$, where $S_i$ are the steps of $L_1$)
or not. 
Note that if $L$ and $L_1$ share at least one common edge, then 
we cannot have that $L\cap L_1=e_1$, since if $A_1$ occurs then 
$e\notin E[\Gnp]$.

{\it Case 1.} If $L$ and  $L_1$ ``intersect cleanly'' then there are $O(h)$ possibilities for where
(i.e., the height at which)
$L\cap L_1$ is located in $L$. Apart from this, by an argument similar to that in 
Case 1 in proof of \cref{L_gamma}, we see that the expected number of such $L$
is at most ${n\choose2}O(h p\alpha_\eps^k/n^2)\le O(k p\alpha_\eps^k)$. 
The compensating factor $1/n^2$ here is due to the fact that 
there are $(v_H-2)h'+2$ vertices, but only 
$\lambda(v_H-2)h'$ edges,  in $L\cap L_1$. 

{\it Case 2.} Otherwise, if $L$ and  $L_1$ do not ``intersect cleanly'' then, 
arguing as in Case 2 in the proof of \cref{L_gamma}, the expected number of such $L$
in this case is $\ll{n\choose2}n^{-\vartheta_\eps}p\alpha_\eps^k$, where 
$\vartheta_\eps>0$ is as defined in \eqref{E_delta}.  

Altogether, 
\[
\frac{p\alpha_\eps^k}{n^2}\sum_{j>1}\P(B_j|A_1)
\le O[(p\alpha_\eps^k)^2(kn^{-2}+n^{-\vartheta_\eps})]
=O[(p\alpha_\eps^k)^2 n^{-\vartheta_\eps}]
\ll1,
\]
since, by \eqref{E_ENk>>1} and \eqref{E_delta},
\begin{equation}\label{E_deltaeps2}
\frac{2}{\lambda}(-1+\lambda\beta_\eps(v_H-2)\log\alpha_\eps)
-\vartheta_\eps
=
-\frac{\vartheta_\eps}{2}<0.\qedhere
\end{equation}
\end{proof}

\subsection{The upper bound}

With \cref{L_2ndMM} at hand, we obtain our upper bound for $p_c$
by an adaptation of the argument found at the end of 
Section 2 in \cite{BBM12}. 

\begin{proof}[Proof of \cref{T_Hmain}]
Let $p=(\alpha/n)^{1/\lambda}$. We show that
for $\alpha>0$ sufficiently large, $\l\Gnp\r_H=K_n$ with high probability.

Let $G=(V,E)$ be a graph. If only $\eps n$ vertices $v\in V$
have degree $d_v\ge\delta (n-1)$, for some $\delta$,  
then $|E|\le [\eps+(1-\eps)\delta]{n\choose2}$. 
Hence, if $|E|>\gamma{n\choose2}$, there is a set
$S\subset V$ of size satisfying $|S|/n\ge(\gamma-\delta)/(1-\delta)$
so that all $v\in S$ have $d_v\ge\delta(n-1)$.

Therefore, by  \cref{L_2ndMM}, for $\alpha>0$
large (and so $\gamma$ close to $1$) with high probability 
there is a set $S$ of size $\Omega(n)$ such that all neighborhoods
$N_v$  in $\l\Gnp\r_H$ of vertices $v\in S$ are larger than $(3/4)n$. 
As a result, all $u,v\in S$
have $|N_u\cap N_v|\ge n/2$. 
Also, for $\alpha$ large enough, all induced
subgraphs of $\l\Gnp\r_H$ of size $n/4$ contain a copy of $K_{v_H-2}$
by Tur\'an's Theorem. 
Hence all edges between vertices in $S$ are in $\l\Gnp\r_H$.

Once a percolating subgraph $S$ of size $\Omega (n)$
has been established, the result follows easily
by sprinkling, as in \cite{BBM12}. 
For completeness, we sketch the argument.
Consider a random graph $\G_{n,p'}$ that is
independent of $\Gnp$ with 
$(\log n)/n\ll p'\ll p$. 
Since $H$ is strictly balanced, such a $p'$ exists, as $p=\Omega(n^{-1/\lambda})$ 
for some $\lambda>1$ (see below 
\cref{D_SB}). 
Due to $(\log n)/n\ll p'$, with high probability, 
in the graph $\G_{n,p'}$, 
all vertices outside of $S$ have
at least $v_H-2$ neighbors in $S$.
Hence, $\l\Gnp\cup \G_{n,p'}\r_H=K_n$ with high probability.
This implies the result, noting that
$\Gnp\cup \G_{n,p'}$ is a random graph 
with edge probability $1-(1-p)(1-p')\sim p$.
\end{proof}

\appendix

\section{Supplementary facts}\label{A_supp}

\subsection{Balanced graphs}\label{A_supp1}
We note here some basic facts about balanced graphs $H$. 
Recall \cref{D_bal} and the definition of 2-balanced graphs below that. 
Also recall that we assume throughout this work that $\delta_H\ge2$ and $v_H\ge4$. 
Hence $e_H\ge v_H$. 

\begin{lemma}
For any graph $H$, we have that $H$ is balanced if and only if 
$H\setminus e$ is 2-balanced for all edges $e\in E[H]$. 
\end{lemma}

\begin{proof}
Suppose that $H\setminus e$ is 2-balanced for all edges $e\in H$. 
Let $F$ be a proper subgraph of $H$ with $v_F\ge3$. 
Let $e\in E[H]\setminus E[F]$. Since $H'=H\setminus e$ is 2-balanced, 
it follows that $(e_F-1)/(v_F-2)\le(e_{H'}-1)/(v_{H'}-2)=\lambda$. 
Thus $H$ is balanced. 

On the other hand, if $H$ is balanced, then for any proper subgraph $F$
of some $H'=H\setminus e$ with $v_F\ge3$, 
$(e_F-1)/(v_F-2)\le\lambda=(e_{H'}-1)/(v_{H'}-2)$.
Thus $H'$ is 2-balanced.  
\end{proof}

\begin{lemma}\label{L_con}
For any graph $H$, if $H$ is balanced then it is connected. 
\end{lemma}

\begin{proof}
Assuming that $H$ is balanced, we show that 
there is at least one edge between any two 
non-empty sets $V_1,V_2$ that partition the 
vertex set of $H$. 
Let $v_i$ and $e_i$ be the number of vertices and edges, respectively,
in the subgraph of $H$ induced by $V_i$, 
and $e_{12}$ the number of edges in $H$
between $V_1$ and $V_2$, so that $e_H=e_1+e_2+e_{12}$. 
If either $v_i\le 2$ 
or $e_1+e_2\le 3$ 
then $e_{12}\ge1$, since $\delta_H\ge2$ and $e_H\ge v_H\ge4$. 
Hence assume that both $v_i\ge3$ and $e_1+e_2\ge4$. 
Then both 
\[
\frac{e_i-1}{v_i-2}\le\lambda=\frac{e_1+e_2+e_{12}-2}{v_1+v_2-2}. 
\]
Taking a weighted average, with weights $v_i-2$, 
it follows that 
\[
\frac{e_1+e_2-2}{v_1+v_2-4}
\le 
\lambda
\]
and so 
\[
e_{12}\ge \left(\frac{v_1+v_2-2}{v_1+v_2-4}-1\right)(e_1+e_2-2)>0.\qedhere
\]
\end{proof}

\providecommand{\bysame}{\leavevmode\hbox to3em{\hrulefill}\thinspace}
\providecommand{\MR}{\relax\ifhmode\unskip\space\fi MR }
\providecommand{\MRhref}[2]{%
  \href{http://www.ams.org/mathscinet-getitem?mr=#1}{#2}
}
\providecommand{\href}[2]{#2}

\end{document}